\documentclass{amsart}

\usepackage{amsmath}
\usepackage{amsfonts}
\usepackage{amssymb}
\usepackage{amsthm}
\usepackage{amsopn}
\usepackage{amscd}
\usepackage{amsxtra}
\usepackage{xypic}
\usepackage[arrow,matrix,curve]{xy}
\usepackage{blkarray}
\usepackage{multirow}
\usepackage{bigdelim}
\usepackage{hyperref}
\usepackage{xcolor}
\usepackage{colortbl}
\usepackage{graphicx}
\usepackage{pstricks}
\usepackage{url}

\newtheorem{theorem}{Theorem}[section]

\newtheorem{proposition}[theorem]{Proposition}

\newtheorem{corollary}[theorem]{Corollary}

\theoremstyle{definition}

\theoremstyle{remark}

\newcommand{\CC}{\mathbb{C}}

\newcommand{\NN}{\mathbb{N}}

\newcommand{\PP}{\mathbb{P}}

\newcommand{\del}{\delta}

\newcommand{\diff}{\mathrm{d}}

\newcommand{\fb}{\mathbf{f}}

\newcommand{\db}{\mathbf{d}}
\newcommand{\qb}{\mathbf{q}}

\newcommand{\Sp}{\, \mathrm{Span}}

\newcommand{\oschu}{\Sigma^\circ}

\newcommand{\Sym}{\operatorname{Sym}}
\newcommand{\Gr}{\operatorname{Gr}}

\newcommand{\Symd}{\Gamma_{n+1}(d)}
\newcommand{\Symdb}{\Gamma_{n+1}(\db)}
\newcommand{\SymL}{S_{\gL'}(\db)}

\newcommand{\gL}{\Lambda}

\newcommand{\calA}{\mathcal{A}}

\newcommand{\fanosp}{Fano scheme }

\newcommand{\fanossp}{Fano schemes }

\newcommand{\qi}{q^{(i)}}
\newcommand{\Qi}{Q^{(i)}}

\makeatletter
\providecommand*{\diff}%
        {\@ifnextchar^{\DIfF}{\DIfF^{}}}
\def\DIfF^#1{%
        \mathop{\mathrm{\mathstrut \Delta}}%
                \nolimits^{#1}\gobblespace
}
\def\gobblespace{%
        \futurelet\diffarg\opspace}
\def\opspace{%
        \let\DiffSpace\!%
        \ifx\diffarg(%
                \let\DiffSpace\relax
        \else
                \ifx\diffarg\[%
                        \let\DiffSpace\relax
                \else
                        \ifx\diffarg\{%
                                \let\DiffSpace\relax
                        \fi\fi\fi\DiffSpace}
\makeatother

\author{Franz Kir\'aly and Paul Larsen}
\address{Machine Learning Group, Technische Universit\"at Berlin, 10623 Berlin, Germany\\ and Discrete Geometry Group, Freie Universit\"at Berlin, 14195 Berlin, Germany}
\email{franz.j.kiraly@tu-berlin.de}
\address{Algebraic Geometry Group, Humboldt-Universit\"at zu Berlin, 10099 Berlin, Germany}
\email{larsen@math.hu-berlin.de}
\title{Fano schemes of generic intersections and machine learning}

\numberwithin{equation}{section}

\begin{document}

\begin{abstract}
We investigate Fano schemes of conditionally generic intersections, i.e. of hypersurfaces in projective space chosen generically up to additional conditions. Via a correspondence between generic properties of algebraic varieties and events in probability spaces that occur with probability one, we use the obtained results on Fano schemes to solve a problem in machine learning.
\end{abstract}
\maketitle

\section{Introduction}
Properties of generic points of a variety, or generic members of a family of varieties, are often much easier to study than their non-generic counterparts. An obvious reason why generic properties are simpler to study is that open sets in algebraic geometry are also dense, and so generic properties are topologically very natural ones to consider. One example of this phenomenon is the family of lines on a hypersurface in projective space, i.e. the hypersurface's Fano scheme of lines: the dimension, smoothness and connectedness of this Fano scheme for a generic hypersurface were determined in \cite{MR510081}, whereas its dimension for arbitrary smooth hypersurfaces is still unknown except in certain cases \cite{MR533324, MR1646558, MR2222727,MR2542223, MR2745753}. The conjectured dimension in general is the content of the \emph{Debarre--de Jong conjecture} \cite{MR1841091}. This paper studies Fano schemes of $k$-planes for intersections of hypersurfaces in $\PP^n$ that are chosen generically up to some additional property, which we call \emph{conditional genericity}. Our main results characterize the dimension of these Fano schemes, with the chosen conditions coming from an application to machine learning.

If generic properties are sometimes the only ones accessible in algebraic geometry, they are in a sense the only ones of importance in statistics and machine learning. To be more precise, given a probability space $(\Omega, \calA, \mu)$, an event $E$ occurs \emph{almost surely} if it occurs with probability one, that is, if $\mu(E) = 1$. Suppose that $\Omega$ is an algebraic variety, and $\mu$ is a continuous measure. If \emph{generic events} in $\Omega$ are taken to be (Zariski) open, measurable subsets, then non-generic events occur with probability zero. Roughly speaking, non-generic events may be ignored in statistics and machine learning because they never occur.

Our application to machine learning involves  \emph{stationary subspace analysis} (SSA) \cite{bunmeikirmul09finding}, a method for multivariate time series analysis that has been applied to brain-computer interface research. We frame the problem and its relation to Fano schemes in Section \ref{sec:app}, but present here one simplified application from brain-computer interface research. Suppose our goal is to isolate a brain signal for a particular task from other signals that vary with time and are irrelevant to this task, e.g. we seek to separate one type of brain activity from other activity related to environmental factors or so-called \emph{alpha oscillations} due to fatigue. The data is divided into $N$ epochs, with each epoch modeled as a random variable. In Section 3, we develop a precise \emph{identifiability} criterion for SSA---i.e. how many epochs are necessary to identify the stationary signal---in terms of the dimension of a certain Fano scheme. 

We now describe the structure of the paper. Section \ref{sec:main} develops results on Fano scheme of various conditionally generic intersections. In Section \ref{sec:commonSubspace} we study the Fano scheme of generic intersections of hypersurfaces under the assumption that all hypersurfaces contain some fixed $k$-plane. Section \ref{sec:lowerRank} restricts to Fano schemes of intersections of quadric hypersurfaces, but conditions the genericity assumption on the quadrics both containing a fixed $k$-plane and having a fixed rank. We conclude with Section \ref{sec:app}, where we relate results of the previous sections to machine learning.

\textbf{Notation and conventions:}
All work below is carried out over the complex numbers, although the results hold for any algebraically closed field. 
Unless otherwise specified, dimension will refer to projective dimension; e.g. $\dim \emptyset = -1.$ For a linear subspace $\gL \subseteq
  \PP^n$ of dimension $k$, we write $[\gL]$ for the corresponding
  element of the Grassmannian $\Gr(k,n)$, and when considering a vector subspace $S$ instead as a projective subpace, we write $\PP \, S$. 

For the vector space of homogeneous polynomials of degree $d$ in $n+1$ variables we write $\Symd$. For a mulitidegree $\db = (d_1, \ldots, d_s)$, the product of the corresponding vector spaces of polynomials will be denoted by $\Symdb$. 

A brief note for the reader coming from outside of algebraic geometry: the word ``scheme'' can safely be replaced with ``variety'' in almost all appearances below. The second author has also prepared a companion set of notes to this paper for those with rudimentary knowledge of algebraic geometry \cite{FanoNotes}.

\textbf{Acknowledgements:}
We would like to thank Bernd Sturmfels for introducing us to one another, for suggesting the connection to Fano schemes, and for his encouragement throughout. We are grateful to Fabian M\"uller for his invaluable help at numerous junctures during this project. We would also like to thank Luke Oeding for his insights on symmetric tensors. Several of the results of this paper were first tested with Macaulay2 \cite{M2} and MATLAB(R)/Octave.

\section{\fanossp of intersections of conditionally generic hypersurfaces}
\label{sec:main}
The $k$th Fano scheme of a projective variety $X$, denoted $F_k(X)$, is
the subscheme of the Grassmannian $\Gr(k,n)$ parametrizing $k$-planes
contained in $X$. If $f^{(1)}, \ldots, f^{(s)}$ are defining
equations of $X$, with $\deg f^{(i)} = d_i$, then we write $X = V(\fb)$, where $\fb = (f^{(1)},
\ldots, f^{(s)})$, and call $\db = (d_1, \ldots, d_s)$ the
multidegree of $\fb$.

The Fano scheme $F_k(V(\fb))$ has been studied in
\cite{MR1654757} under the assumption that $\fb \in \Symdb$ is chosen
generically. Let $n, k,s \in \NN$ and $\db = (d_1, \ldots, d_s) \in \NN^s$, with
all $d_i \geq 2$, and set
\begin{align*}
&\del(n,\db,k)  = (k+1)(n-k) - \binom{\db+k}{k}.
\end{align*}
\begin{theorem}[\cite{MR1654757}]
\label{thm:DM}
Let $[\fb] = [(f_1, \ldots, f_s)] \in \PP\, \Symdb$ be generic, and
suppose that $\db \neq (2)$.
\begin{enumerate}
\item If 
$\del(n,\db,k) < 0$,
 then the
  \fanosp$F_k(V(\fb))$ is empty.
\item If 
$\del(n,\db,k) \geq 0$, 
then the \fanosp
  $F_k(V(\fb))$ is smooth of dimension $\del(n, \db, k)$.
\end{enumerate}
\end{theorem}
\noindent The case of a single quadric is also settled in \cite{MR1654757} using
different argumentation, but
the Fano scheme of a single hypersurface carries little interest for our
application to machine learning (see Section \ref{sec:app}). Therefore
we assume in the remainder that
$\db \neq (2)$.

 We extend these results to
conditionally generic $[\fb] \in \PP \, \Symdb$ in two ways. First, in Section
\ref{sec:commonSubspace}, we take the defining equations to be generic up
to the assumption that each hypersurface $V(f^{(i)})$ contains a fixed
$k$-plane $\gL'$. And second, in Section \ref{sec:lowerRank}, we
restrict to intersections of quadrics that are generic, conditional upon
all quadrics having a fixed rank $r$ and containing a fixed
$k$-plane. Our application to machine learning involves
determining when, under the conditional genericity mentioned above,
there is equality $F_k(V(\fb)) = [\gL']$, i.e. when identifiability is
achieved (see Section \ref{sec:app} for details). Thus determining
when the Fano scheme has dimension 0 is especially important.

\subsection{Generic intersections containing a common subspace}
\label{sec:commonSubspace}
Suppose first that $\del(n,\db,k) \geq 0$. The analogue of statement (2)
in Theorem \ref{thm:DM} for generic $[\fb] \in \PP \,
\Symdb$ such that $\gL' \subseteq V(\fb)$ is a straightforward
corollary.

\begin{corollary}
\label{cor:delgeq0}
Let $\gL' \in \Gr(k,n)$ be generic, and let $[\fb] \in \PP\, \Symdb$
be generic such that $\gL' \subseteq V(\fb)$. If $\del(n, \db,k)
\geq 0$, then $\dim F_k(V(\fb)) = \del(n,\db, k)$.
\end{corollary}
\begin{proof}
Let $\SymL \subseteq \Symdb$ be the subpace consisting of all $s$-tuples
$\fb$ such that $\gL' \subseteq V(\fb)$. This is a codimension
$\binom{\db + k}{k}$ subspace. Consider the incidence correspondence
\begin{equation*}
\xymatrix @C=0.5em{
&I = \{([\fb], [\gL]):  \gL \subseteq V(\fb) \}
\ar[ld]_{p} \ar[rd]^{q}
& \hspace{-10pt}\subseteq  \PP\, \Symdb \times \Gr(k,n)
\\
\PP\, \Symdb
&&\Gr(k,n).
}
\end{equation*}
The proof of Theorem \ref{thm:DM} involves showing that, under the
 conditions of part (2), $p$ is
dominant, and hence by the theorem of the
dimension of the fiber, there exists an open set $U \subset \PP \,
\Symdb$ over which all fibers---which are Fano schemes
$F_k(V(\fb))$---have the expected dimension $\del(n,\db,k)$. The
corollary will follow after showing that for generic $\gL' \in
\Gr(k,n)$, the intersection $\PP \,\SymL \cap U$ is non-empty, and
hence dense in $\PP \, \SymL$.

Set $\Gr = \Gr( \binom{\db + n}{n} - \binom{\db + k}{k} - 1, \binom{\db + n}{n}
- 1)$, so that $[\SymL] \in \Gr$.
Note that $p$ factors as follows: if $\tilde{I} = \{ ([\fb],
[S_{\gL'}]): \fb \in \SymL\} \subseteq \PP \, \Symdb \times \Gr$,
then 
\begin{equation*}
p: I \to \tilde{I}  \stackrel{\tilde{p}}{\to} \PP\, \Symdb
\end{equation*}
But dominance of $p$ implies that $\tilde{p}^{-1}(U)$ is dense in $\tilde{I}$, so for generic
$[\gL']$ it follows that $[\gL']$ is in some fiber of $\tilde{p}$ over
$U$, and hence $\PP \, \SymL \cap U \neq \emptyset$, as required.
\end{proof}

\noindent Next we show that the remaining cases where $\del(n,\db,k) <0$ result
in identifiability.
\begin{theorem}
\label{thm:KL}
Let $[\gL'] \in \Gr(k,n)$ be generic, and let $[\fb] \in \PP\, \Symdb$ be generic such that
$\gL' \subseteq V(\fb)$. If  $\del(n, \db,k)
< 0$, then $F_k(V(\fb)) = [\gL']$.
\end{theorem}
\begin{proof}
For the incidence correspondence $I$ above, the fibers of $q$ are all subspaces of the same dimension in $\PP
\Symdb$. In the situation of Theorem \ref{thm:KL}, however, the incidence
corresponence restricts to those $([\fb], [\gL])$ such that $\gL$ and
$\gL'$ are contained in $V(\fb)$, where $\gL'$ is fixed. The fibers of
$q$ are no longer projective spaces of the same dimension, but rather
their dimension depends on the intersection $\gL \cap \gL'$.

As in the proof of Corollary \ref{cor:delgeq0}, we denote the subpace
of $\PP \, \Symdb$ consisting of $[\fb] \in \PP \, \Symdb$ that vanish on
$\gL'$ by $\PP \,\SymL$. For $[\gL] \in \Gr(k,n)$ such that $\dim \gL \cap \gL' = k'$, $k'
\in \{-1, \ldots, k\}$, we can
choose a basis for $\CC^{n+1}$ such that $\gL = \PP \,\Sp(e_0, \ldots,
e_k)$ and $\gL' = \PP \,\Sp(e_{k-k'}, \ldots, e_{2k - k' })$.  If we replace the original incidence correspondence with $I_{k'} = \{
([\fb], [\gL]) : \gL', \gL \subseteq V(\fb), \dim \gL \cap \gL' =
{k'}\}$, then the non-empty fibers of $q$ over $\Gr(k,n)$ are now projective
spaces of a constant dimension, and the non-empty fibers of $p$ over
$\PP \,\SymL$ are the following subvarieties of $F_k(V(\fb))$,
\begin{equation*}
F_{k,{k'}} = \{ [\gL] \in F_k(V(\fb)): \dim \gL \cap \gL' = {k'} \}.
\end{equation*}

The natural description of the image of the second projection, $q$, is
via Schubert cells. Given a complete flag of projective subspaces $W_i \subseteq \PP^n$,
\begin{equation*}
\emptyset \subsetneq W_0 \subsetneq \ldots \subsetneq W_{n} = \PP^n
\end{equation*}
such that $\gL' = W_k$, and a non-increasing integral vector $\lambda
= ( \lambda_0,
\ldots, \lambda_{k})$ with $0 \leq \lambda_i \leq n-k$, the \emph{Schubert
  cell}\footnote{We use the less-common practice of defining
  Schubert cells via
  projective, rather than affine, dimension.} $\Sigma_\lambda \subset \Gr(k,n)$ is
\begin{equation*}
\Sigma_\lambda = \{[\gL]: \dim \gL \cap W_{n-k  + i -
  \lambda_i} \geq
i \textrm{ for all }i\}.
\end{equation*}
The codimension of $\Sigma_\lambda$ in $\Gr(k,n)$ is $\sum_{i=0}^{k}
\lambda_i$ (see \cite{MR507725} for further details). 

For example, if $n=6$, $k=1$, and ${k'}=0$, then $F_{1,0}(V(\fb))$ is the
intersection of $F_1((\fb))$ with the interior of $\Sigma_{4,0}$;
if instead $k=2$, then $F_{2,0}(V(\fb))$ is obtained by intersecting
with the interior of $\Sigma_{2, 0, 0}$; and if $k=2$ and ${k'}=1$,
then we intersect with the interior of $\Sigma_{3,3,0}$ to obtain
$F_{2,1}(V(\fb))$. In general, for $-1 \leq {k'} \leq k$,
\begin{equation*}
F_{k,{k'}}(V(\fb)) = F_k(V(\fb)) \cap \Sigma^\circ_{\underbrace{\scriptstyle{n-2k + {k'}, \ldots, n-2k +
  {k'}}}_\text{$({k'}+1)$-times}},
\end{equation*}
where we use the common abbreviation of omitting the components of
$\lambda$ equal 0.

We now write the incidence correspondence of interest as
\begin{equation*}
\xymatrix @C=0.5em{
&I_{k'} = \{([\fb], [\gL]):  \gL, \gL' \subseteq V(\fb), \, \dim \gL
\cap \gL' = k'\}
\ar[ld]_{p_{k'}} \ar[rd]^{q}
& \hspace{-10pt}\subseteq  \PP \,\SymL \times \oschu
\\
\PP \,\SymL
&&\oschu.
}
\end{equation*}
To prove Theorem
\ref{thm:KL}, we determine the expected
dimension of $F_{k,{k'}}$ for each ${k'} \in \{-1, \ldots, k\}$, and then
show that under the assumption $\del(n, \db, k) <0$, all expected
dimensions are negative except when $k' = k$, i.e. when $F_{k,k'} = [\gL']$.

The map $q$ is surjective with fibers projective spaces of the same dimension, so $I_{k'}$ is smooth
and irreducible of
codimension $\binom{\db + k}{k} - \binom{\db + k'}{k'}$ in $\PP \,\SymL \times \oschu$, and
\begin{align*}
\dim I_{k, k'} 
& =(k+1)(n-k) - (k'+1)(n-2k + k') \\
& \quad + \binom{\db + n}{n} - 2 \binom{\db +
  k}{k} + \binom{\db + k'}{k'} - 1.
\end{align*}
Subtracting $\dim I_{k'}$ from $\dim \PP \,\SymL$ implies that the expected
dimension of $F_{k,k'}$ is
\begin{equation}
\label{eq:delKL}
\del(n, \db, k,{k'}) = (k-{k'})(n-k+{k'}+1) + \binom{\db + {k'}}{{k'}} -
\binom{\db + k}{k}.
\end{equation}
Here we take $\binom{r}{-1} = 0$ for any $r \in \NN$.

To finish the proof, it suffices to show that all expected dimensions
are negative (except when $k' = k$, when the dimension is 0). We denote the first forward difference of $\del(n,\db, k, k')$ with respect to
$k'$ by $\Delta(k') = \del(n,\db,k, k' + 1) - \del(n,\db,k, k')$, and the
second forward difference by $\Delta^2(k') = \Delta(k'+1) - \Delta(k')$. Then
\begin{align*}
\Delta(k') &= -2 k' -n + 2k -2 + \binom{\db + k'}{k'+1}, \\
\Delta^2(k') & = -2 + \binom{\db + k'}{k' + 2}.
\end{align*}
Since we assume $\db \neq (2)$, $\Delta^2(k')$ is non-negative for
$k' \in \{-1, \ldots, k\}$, and hence $\del(n,\db, k, k')$ is convex
for these $k'$. Note further that $\del(n, \db, k, -1) = \del(n,\db, k)$, and $\del(n,
\db, k, k) = 0$.
By assumption,
$\del(n,\db, k) < 0$, so all expected dimensions except for $k'=k$ are
negative.

It follows that the map $p_{k'}$
cannot be dominant for $k' \in \{-1, \ldots, k-1\}$, and so the
variety $F_{k,k'}$ will be empty for generic $[\fb] \in \PP \,\SymL$. Since
$F_{k,k} = [\gL']$, this completes the proof.

\end{proof}

We conclude by studying the case $\delta(n, \db, k) = 0$. Replacing $\PP \, \Symdb$ by $\oschu$ for $k=-1$ in the proof of Theorem \ref{thm:DM} in
\cite{MR1654757} yields the following (for an expository account of this proof, see
\cite{FanoNotes}).
\begin{corollary}
Let $\gL' \in \Gr(k,n)$ be generic, and let $[\fb] \in \PP \, \Symdb$
be generic such that $\gL' \subseteq V(\fb)$. If $\del(n, \db, k)
= 0$, then $\dim F_{k, -1}(V(\fb)) = 0$.
\end{corollary}
\noindent In particular, there exists in this case a $k$-plane distinct from $\gL'$
(in fact, disjoint from it) contained in $F_k(V(\fb))$. Combining with
Theorem \ref{thm:KL} will give a precise characterization of
identifiability in SSA.

\begin{corollary}
\label{cor:id1}
Let $\gL' \in \Gr(k,n)$ be generic, and let $[\fb] \in \PP \, \Symdb$
be generic such that $\gL' \subseteq V(\fb)$. Then $F_k(V(\fb)) = [\gL']$
if and only if $\del(n,\db, k) < 0$.
\end{corollary}


\subsection{Lower rank quadrics}
\label{sec:lowerRank}
We specialize now to intersections of quadrics,
with all quadrics vanishing on a common $k$-plane $\gL'$, and all with
rank $r$ (meaning their Gram matrices all have rank
$r$). We begin with a result that makes no assumption about a common
linear subspace. Since we only deal with quadrics, we set
\begin{equation*}
\del(n,s,k) = (k+1)(n-k) - s \binom{k+2}{2}.
\end{equation*}
\begin{proposition}
\label{prop:rank}
Let $\qb = (q^{(1)}, \ldots, q^{(s)})$, be generic homogeneous quadratic
forms of rank $r$. If $r \geq 2k +
2$ and
$\del(n,s,k) \geq 0$, then $F_k(V(\qb))$ has dimension $\del(n,s,k)$.
\end{proposition}

\begin{proof}
An immediate corollary of Theorem 2.1 of \cite{MR1654757} (Theorem \ref{thm:DM} above) is that for a $k$-plane $\gL_0 \in
F_k(V(\qb))$, the Jacobian matrix defining the tangent space
$T_{[\gL_0]} F_k(V(\qb))$ is full-rank.
Choose coordinates such that $\gL_0 =
\Sp(e_0, \ldots, e_k)$, and consider the coordinate patch $U_0$ of the
Grassmannian $\Gr(k,n)$ centered at $\gL_0$ with coordinates
$\{x_{a,b}: 0 \leq a \leq n, k+1\leq b \leq n\}$, and corresponding
tangent space coordinates denoted by $X_{a,b}$. For each $i \in \{1,
\ldots, s\}$, let $\qi_{u,v}$, where $0 \leq u \leq v \leq n$, be
coordinates for the vector space of homogeneous quadratic polynomials
(i.e. the entries of the quadratic form's Gram matrix lying on or
above the diagonal---see Figure \ref{fig:gram} for an example). The defining equations for the tangent space of
$F_k(V(\qb))$ at $[\gL_0]$ are 
\begin{align}
\label{eq:tanEqu}
& \qi_{u,k+1} X_{u,k+1} + \ldots + \qi_{u,n} X_{u,n} = 0,\\
\begin{split}
\label{eq:tanEquv}
& \qi_{u,k+1} X_{v,k+1} + \ldots + \qi_{u,n} X_{v,n} \\
& + \qi_{v,k+1} X_{u,k+1} + \ldots + \qi_{v,n} X_{u,n} = 0,
\end{split}
\end{align}
for all $i \in \{1, \ldots, s\}$ and $u,v \in \{0, \ldots, k\}$ with $u<v$.
The first main ingredient of the proof is that these equations depend
only on the coordinates $\qi_{u,v}$ satisfying $0 \leq u \leq k < v
\leq n$.

The second ingredient is a particular choice of coordinates for
quadrics of rank at most $r$, given by all coordinate functions $\qi_{u,v}$
except for those appearing in the bottom right $(n+1-r) \times
(n+1-r)$ corner of the quadric's Gram matrix (see Figure \ref{fig:gram}). Note that, by our
assumption on $r$, these excluded entries are never among the
coordinate functions appearing in the defining equations of the
tangent space (\ref{eq:tanEqu})--(\ref{eq:tanEquv}). For each $i \in
\{1, \ldots, s\}$, these coordinates
are valid on an open subset of $\Gamma_{n+1}(2)$, which we denote by
$U_r$.  We next show  that $U_r \cap S_{\gL_0}$ is non-empty.

To prove this claim, let $\Qi$ be
the Gram matrix of the quadratic form $\qi$. By assumption all $(r+1)\times (r+1)$ minors of $\Qi$
vanish. The resulting equations enable us to write the omitted
coordinate functions as rational expressions in the remaining
coordinate functions, and the locus in $\Gamma_{n+1}(2)$ where the
denominators of these rational expressions do not vanish is precisely
the open set $U_r$ where these coordinates are valid.

Specifically, fix $u,v$ with $r \leq u \leq v \leq n$. Let $\Qi_{uv}$ be the
$(r+1) \times (r+1)$ submatrix obtained by deleting the rows indexed by $\{r, \ldots, n\}
\setminus \{u\}$ and the columns indexed by $\{r, \ldots, n\}
\setminus \{v\}$. The corresponding minor is
\begin{align*}
0=\det \Qi_{uv} &=
\sum_{\sigma \in \Sym(\{0, \ldots, r-1, u\}, \{0,
  \ldots, r-1, v\})} (-1)^\sigma \prod_{k \in \{0, \ldots, r-1, u\}}
\qi_{k, \sigma(k)}
\\
& = \qi_{uv} \sum_{\sigma \in \Sym(\{0, \ldots, r-1\}, \{0, \ldots,
  r-1\})}(-1)^\sigma \sum_{k \in \{0, \ldots, r-1\}} \qi_{k \sigma(k)}
\\
& \quad +\sum_{\substack{\sigma \in \Sym(\{0, \ldots, r-1, u\}, \{0,
  \ldots, r-1, v\}), \\ \sigma(u) \neq v}} (-1)^\sigma \prod_{k \in
\{0, \ldots, r-1, u\}} \qi_{k \sigma(k)},
\end{align*}
which gives the promised rational expression for $\qi_{u,v}$, with
denominator equal to the top-left $r\times r$ minor of $\Qi$, so that
$U_r$ is the complement of the vanishing locus of this minor. Now it is easy to see that $U_r \cap S_{\gL_0}$ is non-empty, since the
quadric with Gram matrix having entries $\qi_{u,v} = 1$ if $u+v = r$ and
$\qi_{u,v} = 0$ otherwise lies in this intersection. 

To finish the proof, Theorem \ref{thm:DM} guarantees that
for generic $\qb = (q^{(1)}, \ldots, q^{(s)})$, the Fano scheme $F_k(V(\qb))$
is of the expected dimension, equivalently, the equations
(\ref{eq:tanEqu})--(\ref{eq:tanEquv}) are full-rank. By the observation
following these equations, it therefore follows that for generic
coordinate choices
$\qi_{u,v}$, $i =1, \ldots, s$ with $0 \leq u \leq k < v \leq n$,
equations (\ref{eq:tanEqu})--(\ref{eq:tanEquv}) are full-rank. Hence a generic intersection of
quadrics of rank at most $r$ that lies in the coordinate patch
described above is also in the open set where the dimension of
$F_k(V(\qb))$ matches the expected dimension. Since the variety
parametrizing symmetric matrices of rank $r$ is irreducible,
the result follows.

\begin{figure}
\label{fig:gram}
\begin{centering}
\psset{unit=0.7cm, linewidth=0.75pt}
 \begin{pspicture}(-2,-2)(2,2)
$Q^{(i)} = \left(\begin{array}{cc|ccccc}
0 & 0 & \qi_{0,2} & \qi_{0,3} & \qi_{0,4} & \qi_{0,5}\\
0&0 & \qi_{1,2} & \qi_{1,3} & \qi_{1,4} & \qi_{1,5} \\ \hline
\qi_{0,2}& \qi_{1,2}& \qi_{2,2} & \qi_{2,3} & \qi_{2,4} & \qi_{2,5} \\
\qi_{0,3}& \qi_{1,3} & \qi_{2,3}&  \qi_{3,3} &
 \qi_{3,4} &  \qi_{3,5} \\
\qi_{0,4} &\qi_{1,4} & \qi_{2,4} &  \qi_{3,4}
& \cellcolor{green!40} \qi_{4,4} & \cellcolor{green!40} \qi_{4,5}\\
\qi_{0,5} & \qi_{1,5} & \qi_{2,5}& \qi_{3,5} &\cellcolor{green!40} \qi_{4,5}
& \cellcolor{green!90} \qi_{5,5}
\end{array}\right)$
\end{pspicture}
\caption{Gram matrix of a quadric vanishing on $\gL_0$ for $k=1$,
  $n=5$. The excluded coordinate functions for $r=4,5$ are depicted in
  increasing shades of green. }
\end{centering}
\end{figure}
\end{proof} 

The proofs of Corollary \ref{cor:delgeq0} and Theorem \ref{thm:KL} require only notational
changes to extend the above arguments to conditional genericity, where the
additional condition is containment of a
fixed $k$-plane $\gL'$.
\begin{corollary}
\label{cor:lowerRank1}
Let $\gL' \in \Gr(k,n)$ be generic, and let $\qb = (q^{(1)}, \ldots,
q^{(s)})$ be generic homogeneous quadratic forms of rank $r$ such that
$\gL' \subseteq V(\qb)$ and $r \geq 2k+ 2$.
\begin{enumerate}
\item If $\del(n,s,k) \geq 0$, then $\dim
  F_k(V(\qb)) = \del(n,s,k)$.
\item If $\del(n,s,k) <0 $, then $F_k(V(\qb)) = [\gL']$. 
\end{enumerate}
\end{corollary}

A direct extension of Corollary \ref{cor:id1} now leads to a
precise determination of an identifiability criterion for SSA.
\begin{corollary}
\label{cor:id2}
Under the assumptions of Corollary \ref{cor:lowerRank1}, $F_k(V(\qb) =
[\gL']$ if and only if $\del(n,s,k) < 0$. 
\end{corollary}

It would be interesting to extend these lower-rank results to other
multidegrees. For our application to machine learning, this
generalization would eliminate the need to assume that only knowledge
of the first two cumulants is available (see Section \ref{sec:app}). There are however two significant obstacles to
extending the techniques used here. First, it is not
known in general if the variety parametrizing lower rank symmetric tensors of
degree $\geq 3$ is irreducible. And second, coordinates like the ones
used for quadrics are not known in general.
\section{Application to stationary subspace analysis}
\label{sec:app}
In this section, we show how the preceding results can be applied in
machine learning. A central task in multivariate time series analysis is to separate out
data coming from different sources, e.g. filtering out noisy data, or
identifying a time-stationary data source from time-varying
sources. A recent method for this second task is
\emph{stationary subspace analysis} (SSA)
\cite{bunmeikirmul09finding}. 

Let $\{X_0^{(1)}, \ldots, X_0^{(m_1)}, \ldots, X_s^{(1)}, \ldots, X_s^{(m_s)}\}$ be
time series data separated into $s+1$ consecutive
epochs, with $X_i^j \in \CC^{n+1}$ for all $i,j$. Each epoch is modeled as a random variable $X_i$. A central
assumption of SSA is that
these data come from a linear superposition of a $k+1$ dimensional stationary signal, and 
an $n-k$ dimensional non-stationary signal. The task of identifying the
stationary signal is now equivalent to finding a projection matrix $P
\in \CC^{(k+1)\times(n+1)}$
such that $P  X_0, \ldots, P  X_s$ are identically
distributed, i.e.
\begin{equation}
\label{eq:proj}
P X_0\sim P X_1\sim  \dots\sim P X_s.
\end{equation}
To make this problem well-defined, we must further assume
that the data $X_i^{j}$--and hence the random variables $X_i$---are general under the assumption that such a projection $P$ exists. An important problem in SSA is to determine the minimal number  $N$ of epochs required to uniquely determine $P$.

The connection to algebraic geometry and Fano schemes arises through
\emph{cumulants}. Let $X$ be a random variable taking values in
$\CC^{n+1}$, then the $d$th cumulant of $X$ is a symmetric tensor of
degree $d$. For example, the first cumulant is the mean vector of $X$, and the
second cumulant is the covariance matrix of $X$. In general, the
$d$th cumulant of $X$ can be represented as a homogeneous polynomial in $n+1$
variables of degree $d$ (see e.g. Appendix 2 of \cite{MR1322960}). For $z
= (z_0, \ldots, z_n)$, we define the $d$th \emph{cumulant polynomial}
$\kappa_{X,d}(z)$ of $X$ by the generating function
\begin{equation*}
g_X(z) = \log \mathbb{E}(\exp (z \cdot X)) = \sum_{d=0}^\infty \kappa_{X,d}(z).
\end{equation*}
If the cumulant generating functions of two random
variables  have
finite radii of convergence, then by taking Fourier transorms it can
be shown that these two random variables are identically distributed if and only if their
cumulants coincide, just as with random variables taking values
in $\CC$. Hence the defining property of
$P$ from Equation (\ref{eq:proj}) translates into an \emph{a priori} infinite
system of homogeneous algebraic equations,
\begin{equation}
\label{eq:kappaP}
P \circ \kappa_{i,d}(z)  = P \circ \kappa_{j,d}(z),
\end{equation}
for all $d \geq 0$ and all $0 \leq i < j \leq s$, where we have
abbreviated $\kappa_{X_i,d}(z)$ by $\kappa_{i,d}(z)$, and the action
of $P$ on
$\kappa_{i,d}(z)$ is induced from its action on $\CC^{n+1}$. We refer to Section
2 of \cite{MR2913722} for more details. To avoid an infinite number of
equations, in SSA it is assumed that the stationary and non-stationary
signals can be separated from one another by considering only the first two
cumulants.

To connect Equations (\ref{eq:kappaP}) to Fano schemes, let $\gL'
\subseteq \CC^{n+1}$ be
the row-span of the matrix $P$. For $1 \leq i  \leq s$, set $f^{(i)} =
\kappa_{i,d} - \kappa_{0,d} \in \Gamma(d)$. Then Equations
(\ref{eq:kappaP}) are satisfied if and only if all polynomials $f^{(i)}$
vanish identically on $\gL'$, which is equivalent to the condition
$[\gL'] \in F_k(V(f^{(1)}, \ldots, f^{(s)}))$. In case $d=1$, the
resulting equations are linear, so by the genericity assumption on the
$X_i$, these equations only reduce the ambient dimension. Hence
Fano schemes of intersections of quadrics is the main case of interest
for SSA. Since the number of epochs tends to be large, the case of two
epochs---i.e. a single quadric---can be reasonably omitted.

The specific application of our results on Fano schemes to SSA is to
give a precise characterization of the number of epochs required to
uniquely identify the projection $P$. An upper bound for the identifiability of the SSA problem was proven in the Appendix of \cite{MR2913722}:
\begin{theorem}
\label{thm:app}
Let $X_0, \ldots, X_s$ be generic random variables in $\CC^{n+1}$
conditional upon the existence of a projection $P \in \CC^{(k+1)
  \times (n+1)}$ such that the first two cumulants of all $P X_i$ coincide. Then $P$ is uniquely
identifiable if $s \geq (n+2)/k$.
\end{theorem}

\noindent The results from Section \ref{sec:main} sharpen the above result and completely solve the SSA identifiability problem for the case of the first two cumulants:
\begin{theorem}
\label{thm:app2}
Let $X_0, \ldots, X_s$ be generic random variables in $\CC^{n+1}$
conditional upon the existence of a projection $P \in \CC^{(k+1)
  \times (n+1)}$ such that the first two cumulants of all $P X_i$ coincide. Then $P$ is uniquely
identifiable if $s \geq 2(n-k)/(k+1)$, and this bound is sharp.
\end{theorem}
\noindent In some SSA problems, the differences of covariance matrices of the
random variables $X_i$ are rank deficient. As long as the rank is not
too low, Corollary \ref{cor:id2} provides a generalization of Theorem
\ref{thm:app2} for these situations.

\bibliographystyle{amsalpha}
\bibliography{bibliography}

\end{document}